\documentclass[11pt]{amsart}

\usepackage[colorlinks=true, pdfstartview=FitV, linkcolor=blue, 
citecolor=blue]{hyperref}

\usepackage{amssymb,amsmath}
\usepackage{bbm}
\usepackage{a4wide}

\DeclareMathAlphabet{\mymathbb}{U}{bbold}{m}{n}

\newtheorem{theorem}{Theorem}[section]
\newtheorem{prop}[theorem]{Proposition}
\newtheorem{lemma}[theorem]{Lemma}     
\newtheorem{fact}[theorem]{Fact}
\newtheorem{coro}[theorem]{Corollary}
\theoremstyle{definition}

\newtheorem{remark}[theorem]{Remark}

\newcommand{\ts}{\hspace{0.5pt}}
\newcommand{\nts}{\hspace{-0.5pt}}
\newcommand{\RR}{\mathbb{R}\ts}
\newcommand{\CC}{\mathbb{C}}

\newcommand{\ZZ}{\mathbb{Z}}
\newcommand{\NN}{\mathbb{N}}

\newcommand{\cA}{\mathcal{A}}

\newcommand{\cE}{\mathcal{E}}

\newcommand{\cO}{\mathcal{O}}

\newcommand{\ee}{\ts\mathrm{e}}
\newcommand{\dd}{\,\mathrm{d}\ts}
\newcommand{\ii}{\ts\mathrm{i}}
\newcommand{\bs}{\boldsymbol}
\newcommand{\trans}{{\scriptscriptstyle{\mathsf{T}}}}
\newcommand{\tri}{_{\nts\triangle}{\!}}
\newcommand{\ad}{{\,\widetilde{\!\mathrm{ad}\!}\,}}
\newcommand{\one}{\mymathbb{1}}
\newcommand{\nix}{\mymathbb{0}}
\newcommand{\Mat}{\mathrm{Mat}}
\newcommand{\imag}{\mathrm{Im\ts}}
\newcommand{\tr}{\mathrm{tr}}

\newcommand{\exend}{\hfill$\Diamond$}
\newcommand{\defeq}{\mathrel{\mathop:}=}

\newcommand{\myfrac}[2]{\frac{\raisebox{-2pt}{$#1$}}
  {\raisebox{0.5pt}{$#2$}}}

\begin{document}

\title[Equal{\ts}-input matrices and time{\ts}-inhomogeneous Markov
  flows]{On the algebra of equal{\ts}-input matrices \\[1.5mm]
  in time{\ts}-inhomogeneous Markov flows}

\author{Michael Baake}
\address{Fakult\"at f\"ur Mathematik, Universit\"at Bielefeld, \newline
       \indent  Postfach 100131, 33501 Bielefeld, Germany}

\author{Jeremy Sumner}
\address{School of Natural Sciences, Discipline of Mathematics,
         University of Tasmania,
    \newline \indent Private Bag 37, Hobart, TAS 7001, Australia}

\begin{abstract}
  Markov matrices of equal-input type constitute a widely used model
  class. The corresponding equal-input generators span an interesting
  subalgebra of the real matrices with zero row sums.  Here, we
  summarise some of their amazing properties and discuss the
  corresponding Markov embedding problem, both homogeneous and
  inhomogeneous in time. In particular, we derive exact and explicit
  solutions for time-inhomogeneous Markov flows with non-commuting
  generator families of equal-input type and beyond.
\end{abstract}


\maketitle

\section{Introduction}

Markov processes are fundamental throughout probability theory and its
applications. They often show up as Markov chains, both in discrete
and in continuous time; see \cite{Norris} for a comprehensive
exposition. When the state space is finite, one most effectively works
with Markov matrices (non-negative matrices with unit row sums) or
Markov semigroups. In the case of continuous time, the formulation is
usually based on a Markov generator, $Q$ say, and the relevant
(commutative) semigroup is written as $ \{ M(t) : t \geqslant 0 \}$
with $M (t) = \ee^{t \ts Q}$. Markov generators have non-negative
off-diagonal entries and zero row sums, and are also known as
\emph{rate matrices} (we will use both terms synonymously).  When only
a single Markov matrix $M$ is given, it is then a natural question
whether it can occur in such a semigroup, which is known as the
\emph{Markov embedding problem}; see \cite{Elfving,King} for early
work and \cite{Davies,BS1} and references therein for some recent
developments.

While this already is a hard problem, and far from being solved in
sufficient generality, it has one drawback in that it only asks for
the embedding into a time-homogeneous process. In reality, there is
often no good reason for the assumption that a time-independent
generator should exist. Further, in the time-dependent case of
commuting matrices, one faces the issue of common eigenvectors, and
hence strong restrictions on the possible equilibria. Indeed, some
experimental evidence \cite{Cox} around non-stationarity thus points
towards the need for time-dependent generators that do \emph{not}
commute with one another. Consequently, one should also look at the
embeddability into a time-inhomogeneous process with, in general,
non-commuting generators. For finite state spaces, a partial answer is
given by Johansen's theorem \cite{Joh73}, which states that any Markov
matrix that is embeddable in such a Markov flow can be approximated
arbitrarily well by the product of \emph{finitely many} homogeneously
embeddable matrices.

While analysing part of this in previous work \cite{BS2,BS3}, it
became clear that there currently is a rather limited understanding of
the structure of Markov flows, and increasingly so with growing
dimension. In particular, when we look at the Cauchy problem
$\dot{M} = M \ts Q$ with $M(0)=\one$, which is the Kolmogorov forward
equation of the underlying process, there are hardly any explicit
results beyond the case where the matrices $Q(t)$ all commute with one
another, in which case the flow is given by
$M(t) = \exp\bigl( \, \int_{0}^{t} Q(\tau) \dd \tau \bigr)$ with
$t\geqslant 0$, as is straightforward to verify. Once again, there is
usually no good reason for the commutativity assumption, and the main
goal of this paper is to derive some solvable cases that go beyond,
where we start from some models that are widely used in
applications. Here, we also aim at a better understanding of the
corresponding Markov flows.

Now, the existence of a solution to the Cauchy problem, as well as its
uniqueness under some standard continuity assumptions on the generator
family, is guaranteed by the Picard--Lindel\"{o}f theorem of ordinary
differential equations (ODEs); see \cite{A,Walter} for classic
sources.  Also, one can use the \emph{Peano--Baker series} (PBS) to
write down a convergent series expansion of the solution, which can be
quite helpful in understanding various aspects of the solution; see
\cite{BSch} and references therein for a summary. In addition, there
is another approach via the \emph{Magnus expansion} (ME), see
\cite{Magnus,ME}, which harvests the observation that, since
$M(0)=\one$ possesses a real matrix logarithm (in fact, the principal
one), this will remain true at least for small values of $t > 0$. So,
one can write $M(t) = \exp \bigl( R(t) \bigr)$ for some interval
$[0, T]$, with $T>0$ and $R(0) = \nix$, and derive an ODE for the
matrix function $R$ from it.

This is possible via some underlying Lie theory, as summarised in
\cite{ME}, and the relevant ODE is based on a series expansion in
terms of iterated matrix commutators (or Lie brackets). As with the
PBS, beyond the trivial case of commuting generators, it is rarely
possible to compute this series explicitly, but there are instances
known where it is; compare \cite{ME} and references therein. It is
perhaps surprising that our featured model class of equal-input
matrices, which we take from applications in phylogenetics
\cite{Steel}, has a sufficiently strong algebraic structure (see
Fact~\ref{fact:ideal} below) that leads to a natural and wide class of
non-commuting families where an explicit solution is also possible. In
fact, already in \cite{BS2}, we highlighted this structure, when we
demonstrated that the equal-input matrices allow for an explicit
formula of the \emph{Baker--Campbell--Hausdorff} (BCH) formula in this
class, where we refer to the \textsc{WikipidiA} for some background on
BCH.  Below, we employ the PBS and the ME to derive some explicit
solutions in closed form. \vspace{0.5mm}

The paper is organised as follows. We collect some notions and
results, in particular around equal-input matrices, in
Section~\ref{sec:prelim}, before we solve the Cauchy problem for the
case that the generator family consists solely of general equal-input
matrices (Section~\ref{sec:EI}).  Generically, these matrices do
\emph{not} commute with one another, but have a sufficiently strong
algebraic structure to enable an explicit solution. A close inspection
of the algebraic structure reveals that even more is possible, due to
a particular ideal within the algebra of zero row sum matrices. This
leads to two more general classes where we obtain a full solution,
which are treated step by step in Section~\ref{sec:general}. As we
proceed, for the sake of readability, we sometimes employ a style of
presentation where we first derive the results prior to their formal
statement as a theorem. Some material on the PBS and the ME is briefly
summarised in the Appendix.

\section{Notation and preliminaries}\label{sec:prelim}

For a row vector $\bs{x} = (x^{}_{1} , \ldots , x^{}_{d}) \in \RR^d$,
we define the \emph{equal-rows matrix}
\[
    C_{\bs{x}} \, \defeq \, \begin{pmatrix}
    x^{}_{1} & \cdots & x^{}_{d} \\ \vdots & \ddots & \vdots \\
    x^{}_{1} & \cdots & x^{}_{d} \end{pmatrix} ,
\]
which has rank $0$ for $\bs{x}=\bs{0}$ and rank $1$ otherwise.  Its
\emph{spectrum} is $\sigma (C_{\bs{x}}) = \{ x, 0, \ldots, 0 \}$ in
multi-set notation, with
\[
    x \, \defeq \, \tr (C_{\bs{x}}) \, = x^{}_{1} + \dots + x^{}_{d} 
\]    
denoting the \emph{summatory parameter} of $C_{\bs{x}}$, while
$\bs{x}$ is called the \emph{parameter vector} of $C_{\bs{x}}$. Note
that $x=0$ is possible for $\bs{x}\ne \bs{0}$. An important relation
is
\begin{equation}\label{eq:C-product}
    C_{\bs{x}} C_{\bs{y}} \, = \, x \ts C_{\bs{y}} \ts ,
\end{equation}
which holds for arbitrary $\bs{x}, \bs{y} \in \RR^d$ and has the
following consequence.

\begin{fact}\label{fact:nilpotent}
  Any matrix\/ $C_{\bs{x}}$ with\/ $x=\tr (C_{\bs{x}})=0$ is
  nilpotent, with\/ $C^{}_{\bs{0}} = \nix$ and nilpotency degree\/ $2$
  in all other cases. When\/ $\bs{x}\ne \bs{0}$, the null space of\/
  $C_{\bs{x}}$ has dimension\/ $d-1$. Whenever\/ $x\ne 0$,
  $C_{\bs{x}}$ is diagonalisable. Otherwise, if\/ $\bs{x}\ne \bs{0}$
  but\/ $d\geqslant 2$ and\/ $x=0$, the Jordan normal form of\/
  $C_{\bs{x}}$ has precisely one elementary Jordan block of the form\/
  $\left( \begin{smallmatrix} 0 & 1 \\ 0 & 0 \end{smallmatrix}
  \right)$ and is diagonal otherwise.
\end{fact}

\begin{proof}
  The first claim follows immediately from Eq.~\eqref{eq:C-product} by
  taking $\bs{y}=\bs{x}$.  When $\bs{x}\ne\bs{0}$, the relation
  $(v^{}_{1}, \ldots , v^{}_{d}) C_{\bs{x}} =\bs{0}$ is satisfied
  precisely by all $\bs{v}$ with $v^{}_{1} + \ldots + v^{}_{d} = 0$,
  which form a subspace of $\RR^d$ of dimension $(d-1)$. When
  $x\ne 0$, hence also $\bs{x}\ne \bs{0}$, we get one extra
  eigenvector for the eigenvalue $x$, and thus diagonalisability of
  $C_{\bs{x}}$. In the remaining case, the consequence on the Jordan
  normal form is clear.
\end{proof}

When $x=1$ and all entries of $\bs{x}$ are non-negative, $C_{\bs{x}}$
is a Markov matrix, but of little interest as it is singular for
$d\geqslant 2$.  A more interesting set of Markov matrices for
$d\geqslant 2$, with numerous applications in phylogeny \cite{Steel}
for instance, is defined by
\begin{equation}\label{eq:def-Markov-EI}
     M_{\bs{x}}  \, = \, (1-x)\one + C_{\bs{x}}
\end{equation}
when $\bs{x} \geqslant \bs{0}$ (meaning $x_i \geqslant 0$ for all
$1\leqslant i \leqslant d$) with $1 + x_{i} \geqslant x $ for all $i$.
By slight abuse of notation, $x$ is called the \emph{summatory
  parameter} of $M_{\bs{x}}$.  The condition that all row sums are $1$
is satisfied automatically. These are the well-known
\emph{equal-input} Markov matrices, which have $d$ degrees of freedom
and form a closed convex set.  Let us describe them in a little more
detail. The non-negativity condition of the matrix elements restricts
$x$ to the interval $\bigl[ 0, \frac{d}{d-1}\bigr]$, with the maximal
value being attained at
$x^{}_{1} = x^{}_{2} = \ldots = x^{}_{d} = \frac{1}{d-1}$.  This
implies the following result, where $\bs{1} \defeq (1, \ldots , 1)$
and $\bs{e}_i$ denotes the standard unit row vector with $1$ in
position $i$ and $0$ everywhere else; see \cite[Lemma~2.8]{BS2}.

\begin{fact}\label{fact:ei-convex}
  The equal-input Markov matrices for fixed\/ $d\geqslant 2$ form a\/
  $d$-dimensional convex set that is closed and has\/ $d+2$ extremal
  elements, namely the matrices\/ $C_{\bs{e}^{}_i}$ with\/
  $1 \leqslant i \leqslant d$ together with\/ $\one$ and\/
  $\frac{1}{d-1} \ts ( C_{\bs{1}} - \one) $.  \qed
\end{fact}

The matrices $M_{\bs{x}}$ where all entries of $\bs{x}$ are equal are
called the \emph{constant-input} matrices. Note that in
Eq.~\eqref{eq:def-Markov-EI}, independently of whether $M_{\bs{x}}$ is
Markov or not, one always has the spectrum
$\sigma (M_{\bs{x}}) = \{ 1, 1 - x, \ldots , 1 - x \}$ and thus
$\det (M_{\bs{x}}) = (1-x)^{d-1}$.

Of particular interest are the equal-input counterparts with zero row
sums, as defined by
\begin{equation}\label{eq:def-Q}
    Q_{\bs{x}} \, = \, - x \ts \one + C_{\bs{x}} \ts ,
\end{equation}
which satisfy $\lambda Q_{\bs{x}} = Q_{\lambda\bs{x}}$ for all
$\lambda\in\RR$ and $\bs{x}\in\RR^d$. These matrices all lie in the
non-unital, real matrix algebra of zero row sum matrices,
\[
  \cA^{}_{\ts 0} \, = \, \cA^{(d)}_{\, 0} \, \defeq \, \bigl\{ A \in
  \Mat (d,\RR) : \textstyle{\sum_{j=1}^{d}} A_{ij} = 0 \text{ for all }
  1\leqslant i \leqslant d \ts \bigr\} .
\]

In fact, all matrices of the form \eqref{eq:def-Q},
\[
  \cE^{}_{0} \, = \, \cE^{(d)}_{\, 0} \, \defeq \, \big\{ Q^{}_{\bs{x}} :
  \bs{x} \in \RR^d \big\} \ts ,
\]
form a $d$-dimensional subalgebra of $\cA^{(d)}_{\, 0}$ with
interesting properties.\footnote{From now on, whenever the dimension
  is arbitrary but fixed, we shall simply write $\cA^{}_{\ts 0}$ and
  $\cE^{}_{0}$.}  Since $\nix \in \cE^{}_{0}$, the vector space
property of $\cE^{}_{0}$ is clear, while \eqref{eq:C-product} and
\eqref{eq:def-Q} together give
\begin{equation}\label{eq:Q-product}
    Q_{\bs{x}} Q_{\bs{y}} \, = \, - y \ts Q_{\bs{x}} \ts ,
\end{equation}
which implies that $\cE^{}_{0}$ is an algebra. It is not unital,
because it has no two-sided unit, though it has right units \cite{CS},
which are precisely the matrices $Q_{\bs{x}}$ with $x=-1$, as follows
from \eqref{eq:Q-product}.  Within $\cE^{}_{0}$, one also has the
special class of constant-input matrices that are the multiples of
$J_{d} -\one $, where $J_d = \frac{1}{d} \ts C^{}_{\bs{1}}$.  They
will show up a number of times. Another important subset, comprising
nilpotent elements only, is
\[
  \cE^{\prime}_{0} \, \defeq \, \{ Q_{\bs{x}} \in \cE^{}_{0} : x=0 \}
  \, = \, \{ C_{\bs{x}} : \bs{x} \in \RR^d \text{ and } \tr
  (C_{\bs{x}}) = 0 \} \ts ,
\]
which is a subalgebra of $\cA^{}_{0}$. Clearly, for any
$C_{\bs{x}} \in \cE^{\prime}_{0}$, one has
$A \ts C_{\bs{x}} = \nix \in\cE^{\prime}_{0}$ for all $A\in
\cA^{}_{0}$. Also, if we write $A\in\cA^{}_{\ts 0}$ as a collection of
column vectors,
$A = (\bs{a}^{\trans}_{1}, \ldots, \bs{a}^{\trans}_{d})$, we get
$C_{\bs{x}} A = C_{\bs{y}}$ with
$y^{}_i = \bs{x} \cdot \bs{a}^{\trans}_{i} $.  Since
$\tr (C_{\bs{x}} A) = \tr (A \ts C_{\bs{x}}) = \tr (\nix) = 0$, we see
that $C_{\bs{y}}\in\cE^{\prime}_{0}$. Further, for elements from
$\cE^{\prime}_{0}$, Eq.~\eqref{eq:C-product} simplifies to
$C_{\bs{x}} C_{\bs{y}}=\nix$. Together, this shows the following.

\begin{fact}\label{fact:ideal}
  The subalgebra\/ $\cE^{\prime}_{0}$ is a two-sided nil ideal, both
  in\/ $\cE^{}_{0}$ and in\/ $\cA^{}_{0}$.  Further, for any\/
  $A \in \cA^{}_{\ts 0}$ and $C \in \cE^{\prime}_{0}$, one has
  $A \ts C =\nix$.  \qed
\end{fact}

A special case of \eqref{eq:Q-product} is
$Q^{2}_{\bs{x}} = -x \ts Q^{}_{\bs{x}}$, and hence
$Q^{n}_{\bs{x}} = (-x)^{n-1} Q^{}_{\bs{x}}$ for $n\in\NN$, which
results in a simple formula for its time-scaled exponential,
\begin{equation}\label{eq:expo-1}
\begin{split}
  M(t) \, & \defeq \,\ee^{\ts t \ts Q_{\bs{x}}} \, = \,
    \one + \sum_{n=1}^{\infty}
    \myfrac{t^n}{n!} \ts (-x)^{n-1} Q_{\bs{x}} \, = \, \one + 
    \myfrac{1 - \ee^{-tx}}{x} \ts Q_{\bs{x}} \\ 
    & \, = \, \ee^{-tx} \ts \one + \myfrac{1-\ee^{-tx}}{x} 
    C_{\bs{x}} \, = \, M_{\bs{c}(t)} \ts ,
\end{split}
\end{equation}
with $\bs{c}(t) = \frac{1-\ee^{-tx}}{x}\ts \bs{x}$ and
$c(t) = 1- \ee^{-tx}$. This formula also holds when $x=0$, then with
the appropriate limits via de l'Hospital's rule. So, each $M(t)$ is an
equal-input matrix with row sum $1$ and (time-dependent) summatory
parameter $c (t) = 1-\ee^{-tx}$. In particular, when $Q_{\bs{x}}$ is a
Markov generator, which happens if and only if
$\bs{x}\geqslant \bs{0}$, the set $ \{ M(t) : t\geqslant 0 \}$ defines
a (commutative) semigroup of Markov matrices.  Since $M(0)=\one$, it
is actually a monoid.

Let us mention an interesting property of equal-input matrices in the
context of the BCH formula.  Given
$Q_{\bs{x}}, Q_{\bs{y}}\in\cE^{}_{0}$, they will generally not
commute, and neither will their exponentials. However, the product
$\ee^{Q_{\bs{x}}} \ee^{Q_{\bs{y}}}$ is still of equal-input type and
has a real logarithm, which can be given explicitly as a linear
combination in $Q_{\bs{x}}$ and $Q_{\bs{y}}$; see
\cite[Thm.~2.15]{BS2}.  To formulate it, we define the positive
function $h$ on $\RR$ by $h(u) = \frac{1-\ee^{-u}}{u}$ for $u\ne 0$
and $h(0)=1$, which is the continuous extension to $u=0$.

\begin{fact}\label{fact:EI-BCH}
  Let\/ $Q_{\bs{x}} , Q_{\bs{y}} \in \cE^{}_{0}$. Then, the product of
  their exponentials has a unique real logarithm of equal-input type.
  The latter is the principal matrix logarithm and reads
\[
   \log \bigl( \ee^{Q_{\bs{x}}} \ee^{Q_{\bs{y}}}\bigr) \, = \,
   \myfrac{1}{h(x+y)} \bigl( h(x) \ts \ee^{-y} Q_{\bs{x}} + h(y) 
   \ts Q_{\bs{y}}\bigr).
\]
\end{fact}

\begin{proof}
  Since $\ee^{Q_{\bs{x}}} = \one + h(x) \ts Q_{\bs{x}}$ and
  $\ee^{Q_{\bs{y}}} = \one + h(y) Q_{\bs{y}}$ by \eqref{eq:expo-1},
  with our previous convention for $x$ and $y$, one finds
\[
  \ee^{Q_{\bs{x}}} \ee^{Q_{\bs{y}}} \, = \, \one + h(x) \ts Q_{\bs{x}}
  + h(y) \ts Q_{\bs{y}} + h(x) h(y) \ts Q_{\bs{x}} Q_{\bs{y}} \,
  \overset{\eqref{eq:Q-product}}{=} \, \one + h(x) \ts \ee^{-y}
  Q_{\bs{x}} + h(y) \ts Q_{\bs{y}} \, = \, \one + Q_{\bs{z}}
\]  
  with the new parameters 
\[
   \bs{z} \, = \, h(x) \ee^{-y} \bs{x} + h(y) \bs{y}
   \qquad \text{and} \qquad z \, = \, 1 - \ee^{-(x+y)} .
\]  
Recall that $Q^{n}_{\bs{z}} = (-z)^{n-1} Q^{}_{\bs{z}}$ holds for
$n\in\NN$.  Then, for $\lvert z \rvert < 1$, which holds when
$x+y\geqslant 0$, we can compute
\[
  \log ( \one + Q_{\bs{z}}) \, = \sum_{n\geqslant 1}
  \myfrac{z^{n-1}}{n} \ts Q_{\bs{z}} \, = \, \myfrac{ \log (1-z)}{-z}
  \ts Q_{\bs{z}} \, = \, \myfrac{1}{h (x+y)} \ts Q_{\bs{z}} \, = \,
  Q_{\bs{z}/h(x+y)}\ts ,
\]
which is of the form claimed. This now has to be extended to all $x$
and $y$.
   
Since $\sigma (\one + Q_{\bs{z}}) = \{ 1, \ee^{-(x+y)}, \ldots ,
\ee^{-(x+y)} \} \subset \RR^{}_{+}$, we immediately know from Culver's
theorem \cite[Thm.~1]{Culver} that a real logarithm exists. In
particular, the principal matrix logarithm is well defined for these
matrices, which can be given in integral form \cite[Thm.~11.1]{Higham}
as
\[
    \log (\one + Q_{\bs{z}}) \, =  \int_{0}^{1} Q_{\bs{z}}
    \bigl( \one + \tau \ts Q_{\bs{z}} \bigr)^{-1} \dd \tau \ts .  
\]
The resulting matrix is again $Q_{\bs{z}/h(x+y)}$, as follows from our
previous calculation in conjunction with a standard analytic
continuation argument. Further, the claimed uniqueness follows from
the observation that $Q_{\bs{a}} = Q_{\bs{b}}$ in \eqref{eq:def-Q} is
only possible for $\bs{a}=\bs{b}$.
\end{proof}

The real logarithm of $\one + Q_{\bs{z}}$ in the last proof will not
be unique when $d>3$, as there are degeneracies in the spectrum, but
no other real logarithm can be of equal-input form, which follows from
\cite[Lemma~2.14]{BS2}.

This is a rare case of non-commuting matrices where one can give a
closed expression for the BCH formula.\footnote{We refer to the
  \textsc{WikipediA} entry on the BCH formula for a suitable summary
  with references.}  Indeed, via this route, one would find
\[
    \log \bigl( \ee^{Q_{\bs{x}}} \ee^{Q_{\bs{y}}}\bigr) \, = \,
    \Bigl( 1 - \myfrac{y}{2} + \myfrac{y (y-x)}{12} +
    \myfrac{xy^2}{24} + \ldots \Bigr) Q_{\bs{x}}  \ts + \ts
    \Bigl( 1+ \myfrac{x}{2} + \myfrac{x (x-y)}{12} -
    \myfrac{x^2 y}{24} + \ldots \Bigr) Q_{\bs{y}} 
\]
where the dots indicate higher order terms that would stem from
fourfold or higher commutators in the BCH expansion. The Taylor
expansion of the formula in Fact~\ref{fact:EI-BCH} to third order
agrees with this, as it must, because no other logarithm of
equal-input type can exist.

\section{Inhomogeneous flow}\label{sec:EI}

Consider the Cauchy problem defined by $\dot{M} = M Q$ with
$M(0) = \one$, where $Q = Q(t)$ with $t\geqslant 0$ is a (possibly
piecewise) continuous family of equal-input matrices from
$\cE^{}_{0}$.

\begin{lemma}\label{lem:PB-step}
  Assume that\/ $Q(t) \in \cE^{}_{0}$ for all\/ $t\geqslant 0$, and
  that\/ $Q(t)$ is piecewise continuous. Then, the Cauchy problem
\[
     \dot{M} \, = \, M Q   \quad \text{with}  \quad M(0) \, = \, \one
\]   
has a unique solution, where each\/ $M(t)$ with\/ $t\geqslant 0$ is a
real matrix with row sum\/ $1$ that has equal-input form.  Further, if
all\/ $Q(t)$ are rate matrices, $\{ M(t) : t\geqslant 0 \}$ is a flow
of equal-input Markov matrices.
\end{lemma}

\begin{proof}
  Assume first that $Q(t)$ is continuous, but not necessarily of
  equal-input type. When $[Q(t), Q(s)]=\nix$ for all $t,s\geqslant 0$,
  we get the solution in closed form as
\begin{equation}\label{eq:closed-form}
     M(t) \, = \, M(0) \exp \Bigl( \, \int_{0}^{t} Q (\tau) \dd \tau \Bigr) 
      \, = \, \exp \Bigl( \, \int_{0}^{t} Q (\tau) \dd \tau \Bigr) 
\end{equation}
by classic ODE theory \cite{A,Walter}.  Within $\cE^{}_{0}$,
Eq.~\eqref{eq:def-Q} implies that the easy case of commuting matrices
only occurs when $Q(t) = \mu (t) \ts Q^{}_{0}$ with a fixed
$Q^{}_{0}$.  Then, the solution \eqref{eq:closed-form} simplifies to
$M(t) = \exp \bigl( u(t) \ts Q^{}_{0} \bigr)$ with
$u(t) = \int_{0}^{t} \mu (\tau) \dd \tau$, and the claim on the
equal-input nature of $M(t)$ is obvious from the fact that $Q^{2}_{0}$
is a scalar multiple of $Q^{}_{0}$.

Alternatively, one can see it from the convergent series expansion of
the matrix exponential in \eqref{eq:closed-form}. Indeed, when all
$Q(t)$ lie in $\cE^{}_{0}$, we have
$\int_{0}^{t} Q(\tau) \dd \tau \in \cE^{}_{0}$ for every
$t\geqslant 0$, and if all $Q(\tau)$ are rate matrices, then so is the
integral. Since the algebra $\cE^{}_0$ is closed under multiplication
and under taking limits, the equal-input structure is preserved and
then inherited by $M(t)$.

In the general case with non-commuting $Q(t)$, we can represent the
solution of the Cauchy problem by the PBS \cite{BSch}; see \cite{BS3}
for a formulation that is tailored to our present setting, and the
Appendix for a brief summary. Each summand of the PBS (except the
first, which is $\one$) lies in $\cE^{}_{0}$. As the series is
compactly converging, it then admits the same line of conclusions.

When $Q(t)$ is piecewise continuous, the locations of the potential
jumps are isolated, and one can solve the Cauchy problem for each
continuity stretch individually, then with the last value of $M(t)$ as
the new initial condition, modifying \eqref{eq:closed-form} or the PBS
in the obvious way. Harvesting the semigroup property of the (convex)
set of equal-input Markov matrices from \cite{BS2}, we can put the
pieces together and obtain the claimed properties for the entire flow.
\end{proof}

With this result, we can use the parametrising vectors to reduce the
problem to an ODE for vectors in $\RR^d$ as follows. Assume that we
have $Q(t) = Q_{\bs{q} (t)} = -q(t) \ts \one + C_{\bs{q} (t)}$ for
$t\geqslant 0$, with $\bs{q} (t)$ being piecewise continuous.  By
Lemma~\ref{lem:PB-step}, we know that the solution exists and must be
of the form
$M(t) = M_{\bs{x} (t)} = (1-x(t)) \ts \one + C_{\bs{x} (t)}$, for all
$t\geqslant 0$. Clearly, we then have
$\dot{M} (t) = - \dot{x} (t) \ts \one + C_{\dot{\bs{x}} (t)}$, which
is to be compared with
\[
   M(t) \ts Q(t) \, = \, \bigl( x(t) - 1 \bigr) q(t) \ts \one +
   C_{\bs{q} (t) - q (t) \bs{x} (t)} \ts ,
\]
as follows from a short calculation that uses
$C_{\bs{x}} C_{\bs{q}} =x \ts C_{\bs{q}}$. Since
$\one \not\in \cE^{}_{0}$, this leads to the simpler Cauchy problem
\[
    \dot{\bs{x}} + q \ts \bs{x} \, = \, \bs{q}
    \quad \text{with} \quad \bs{x} (0) = \bs{0} \ts .
\]
This inhomogeneous, linear first order ODE can be solved by the
standard variation of constants method \cite[Thm.~11.13 and
Rem.~11.14]{A}, which results in
\begin{equation}\label{eq:vec-sol}
  \bs{x} (t) \, = \, \exp \Bigl( - \! \int_{0}^{t} \!
  q(\rho) \dd \rho\Bigr) \int_{0}^{t} \! \bs{q} (\tau)
  \exp \Bigl( \, \int_{0}^{\tau} \! q(\sigma)
  \dd \sigma \Bigr) \dd \tau \ts .
\end{equation}
Clearly, this gives $\bs{x} (0) = \bs{0}$, and also the scalar
counterpart of the vector-valued ODE, namely $\dot{x} + q \ts x = q$
together with $x(0)=0$, which is also needed to satisfy the original
Cauchy problem.  We have thus established the following result.

\begin{prop}
   The Cauchy problem of Lemma~$\ref{lem:PB-step}$ has the solution\/
   $M(t) = M_{\bs{x} (t)}$ for\/ $t\geqslant 0$, with the\/ $\RR^d$-valued 
   vector function\/ $\bs{x}$ from \eqref{eq:vec-sol}.    \qed
\end{prop}

Here, we have $\det ( M(t)) = (1 - x(t))^{d-1}$, where $x(0)=0$
matches $\det(M(0)) = \det(\one)=1$. Since
$\frac{\dd}{\dd t} \det (M(t)) = \det (M(t)) \ts \tr (Q(t))$ by
Liouville's theorem \cite[Prop.~11.4]{A}, we get
\begin{equation}\label{eq:Liouville-1}
  \det (M(t)) \, = \, \exp \Bigl( \, \int_{0}^{t} \! \tr (Q(\tau))
  \dd \tau \Bigr),
\end{equation}
which never vanishes. But this implies that $x(t)$, which starts at
$0$, can never take the value $1$, as this would make $M(t)$
singular. With $x(0)=0$, this implies $x(t) < 1$ for all
$t\geqslant 0$. Note that all eigenvalues of $M(t)$ are positive real
numbers in this case.

\begin{remark}\label{rem:c-values}
  Liouville's theorem actually implies that $\det (M(t))$ is either
  identically $0$ or never vanishes \cite[Cor.~11.5]{A}. Indeed, when
  we admit more general initial conditions, the formula from
  \eqref{eq:Liouville-1} is simply replaced by
\[
     \det (M(t)) \, = \,  \det (M(0))
     \exp \Bigl( \, \int_{0}^{t} \! \tr (Q(\tau)) \dd \tau \Bigr) ,
\] 
where $Q(t) = -q(t) \one + C_{\bs{q} (t)}$ with
$\tr (C_{\bs{q}(t)}) = q(t)$, hence $\tr(Q(t)) = - (d-1) \ts q(t)$.
In particular, the exponential factor is strictly positive for all
$t\geqslant 0$.  Recall that $\det(M(0)) =
(1-x(0))^{d-1}$. Consequently, the three possible cases now are
$x(0) < 1$ (which means $\det(M(t))>0$ for all $t\geqslant 0$ and
includes the case treated above), $x(0)=1$ (forcing
$\det(M(t)) \equiv 0$), and $x(0)>1$. The last case either implies
$\det(M(t))>0$ or $\det(M(t))<0$ for all $t\geqslant 0$, depending on
whether $d$ is odd or even, respectively. This distinction naturally
occurs in the treatment of equal-input matrices \cite{BS1}. If the
determinant is negative, it must stay so, which is consistent with the
grading of the monoid of equal-input Markov matrices by the sign of
$1-x$, and thus by the determinant for even $d$, as detailed in
\cite[Prop.~2.6]{BS2}.  \exend
\end{remark}

When all $Q(t) = Q_{\bs{q}(t)}$ are rate matrices, which is equivalent
with $\bs{q} (t) \geqslant \bs{0}$, we also know that
$\bs{x} (t) \geqslant \bs{0}$ together with $x(0)=0$ from
\eqref{eq:vec-sol}, and then get the stronger inequality
$\lvert x (t) \rvert < 1$ for all $t\geqslant 0$.  Consequently, the
spectral radius of $A (t) = M(t) - \one$ is always less than $1$, and
\[
    R(t) \, \defeq \, \log \bigl( \one + A(t) \bigr) \, =
    \sum_{n=1}^{\infty} \myfrac{(-1)^{n-1}}{n} A(t)^n
\]
is a convergent series. Observing that $A(t)^n = (- x(t))^{n-1} A(t)$
holds for all $n\in\NN$, one finds
\begin{equation}\label{eq:real-log}
  R(t) \, = \, \frac{\log \bigl( 1- x(t)\bigr)}{- x(t)} A(t) \, = \,
  \frac{-\log \bigl( 1 - x(t) \bigr)}{x(t)}
  \ts\ts Q_{\bs{x} (t)}
\end{equation}
with $R(0)=\nix$ and $R(t) = \nix$ whenever $x(t)=0$, via an argument
of de l'Hospital type, because the latter case is only possible for
$\bs{x} (t) = \bs{0}$ under our assumption that $Q$ is a family of
rate matrices. Further, $R(t)$ is a Markov generator if
$\bs{x}(t) \geqslant \bs{0}$, as $Q_{\bs{x} (t)}$ then is a rate
matrix and the prefactor is non-negative. Here, $R(t)$ is a real
logarithm of the solution function, hence $M(t) = \ee^{R(t)}$, and it
is actually the \emph{principal logarithm}; see \cite{Gant,Higham,HJ} for
background.

More generally, the principal matrix logarithm exists for all (real or
complex) matrices with positive spectrum, and is defined as the unique
logarithm whose eigenvalues all lie in the strip
$\{ z \in \CC: -\pi < \imag(z) < \pi \}$. One general formula follows
from \cite[Thm.~11.1]{Higham} and reads
\[
    \log \bigl(M(t) \bigr) \, = \int_{0}^{1} (M(t) - \one)
       \bigl( \tau (M(t) - \one) + \one \bigr)^{-1} \dd \tau \ts ,
\]
which is applicable for all of the above cases with $x(t)<1$. It is a
consequence of analytic continuation that the formula from
\eqref{eq:real-log} is the correct one also in this more general case.

Since $M(t)$ is of equal-input form, its spectrum will generally be
degenerate (certainly for $d\geqslant 3$). In view of
Fact~\ref{fact:nilpotent}, $M(t)$ can never be cyclic for
$d\geqslant 3$ (meaning that characteristic and minimal polynomial
cannot agree), though for $d=3$ with $x=0$, we can have (non-Markov)
equal-input matrices with Jordan normal form
$1 \oplus \left( \begin{smallmatrix} 1 & 1 \\ 0 & 1 \end{smallmatrix}
\right)$, in which case the real logarithm is still unique. In
general, however, there will be other real logarithms, compare
\cite{Culver}, but none of equal-input form.  Putting the pieces
together, the derived result reads as follows; see \cite{BS2} for
previous results on the embedding problem.

\begin{theorem}
  The Cauchy problem of Lemma~$\ref{lem:PB-step}$ defines a forward
  flow\/ $\{ M(t) : t \geqslant 0 \}$ of equal-input matrices with
  unit row sums, where each\/ $M(t)$ possesses a real logarithm. In
  particular, we have\/ $M(t) = \ee^{R(t)}$ with the matrix function\/
  $R$ from \eqref{eq:real-log} and \eqref{eq:vec-sol}, where $R(t)$ is
  the unique real logarithm of\/ $M(t)$ of equal-input form.
   
  Further, when all\/ $Q(t)$ are Markov generators $(\nts$of
  equal-input type$\ts )$, so\/ $Q(t) = Q_{\bs{q}(t)}$ with\/
  $\bs{q} (t) \geqslant \bs{0}$ for all\/ $t\geqslant 0$, the forward
  flow consists of Markov matrices only. Then, every $R(t)$ is a
  Markov generator of equal-input type as well, and each individual\/
  $M(t)$ is also embeddable in a time-homogeneous Markov semigroup
  that is generated by an equal-input Markov generator.  \qed
\end{theorem}

Let us look at two particular limiting cases that we will need again
later for a comparison of our two different example classes at their
intersection. First, if we assume $Q(t) \in \cE^{\prime}_{0}$ for all
$t\geqslant 0$, we have $Q(t) = Q_{\bs{q}(t)} = C_{\bs{q} (t)}$ with
$q(t) \equiv 0$. Then, \eqref{eq:vec-sol} simplifies to
$\bs{x} (t) = \int_{0}^{t} \bs{q} (\tau) \dd \tau$ with
$x(t) \equiv 0$, and \eqref{eq:real-log} becomes
\begin{equation}\label{eq:trivial-case}
   R(t) \, = \, Q_{\bs{x} (t)} \, = \, C_{\bs{x} (t)}
\end{equation}
after an application of de l'Hospital's rule. This is one special
limiting case.

Next, and more generally, let us write $R(t)$ from \eqref{eq:real-log}
in a different way. Each $Q\in \cE^{}_{0}$ can uniquely be written as
a sum of a constant-input matrix (compare \eqref{eq:def-Q} and the
paragraph before it) and one from the ideal $\cE^{\prime}_{0}$, namely
as $Q = Q_{\bs{q}} = q J_{d} + C_{\bs{r}}$ with
$J_{d} = \frac{1}{d} C_{\bs{1}}$ and
$\bs{r} = \bs{q} - \frac{q}{d} \bs{1}$.  Setting $Q^{}_{0} = J_{d}$,
we now write
\begin{equation}\label{eq:decomp}
     Q(t) \, = \, Q_{\bs{q}(t)} \, = \, \mu (t) \ts Q^{}_{0} + C_{\bs{r} (t)}
\end{equation}
with the scalar function $\mu (t) = q(t)$ and
$\bs{r} (t) = \bs{q} (t) - \frac{q(t)}{d} \bs{1}$. In particular, one
then has $r(t) \equiv 0$. Setting
$u (t) = \int_{0}^{t} \mu (\tau) \dd \tau$, so that $u(0)=0$,
Eq.~\eqref{eq:vec-sol} simplifies to
\[
  \bs{x} (t) \, = \, \ee^{- u(t)} \int_{0}^{t} \ee^{u(\tau)} \bs{q}
  (\tau) \dd \tau .
\]
Its scalar counterpart then is $x(t) = 1 - \ee^{- u(t)}$, where we
used that
$\int_{0}^{t} \ee^{u (\tau)} \mu (\tau) \dd \tau = \ee^{ u(t)} - 1$.
With $f(x) = \frac{x}{\ee^x - 1}$, which will appear many times from
now on, Eq.~\eqref{eq:real-log} turns into
\begin{equation}\label{eq:weighted-integral}
  R (t) \, = \, f (u(t)) \ts \ee^{u(t)} Q_{\bs{x} (t)}
  \, = \, f(u(t)) \int_{0}^{t}
  \ee^{u(\tau)} Q_{\bs{q} (\tau)} \dd \tau \, = \, f(u(t)) \int_{0}^{t}
  \ee^{u(\tau)} Q(\tau) \dd \tau .
\end{equation}
This gives $R (t)$ as a weighted integral over the original generator
family from \eqref{eq:decomp}, which thus is a direct generalisation of
Fact~\ref{fact:EI-BCH}.

\section{Two generalisations and their exact
  solutions}\label{sec:general}

The decomposition of Eq.~\eqref{eq:decomp} suggests that a little more
might be possible than such a sum with $Q^{}_{0}$ a constant-input
matrix. Indeed, we will establish that the algebraic structure of
equal-input matrices from Fact~\ref{fact:ideal} is strong enough to
allow for a significant extension. So, let $Q^{}_{0}$ now be
\emph{any} fixed rate matrix, and consider the matrix family defined
by
\begin{equation}\label{eq:pert-gen}
    Q(t) \, = \, \mu(t) \ts Q^{}_{0} + C_{\bs{q} (t)} 
\end{equation}
subject to the assumption that $q(t) = \tr (C_{\bs{q} (t)}) = 0 $ for
all $t\geqslant 0$, so $C_{\bs{q} (t)} \in \cE^{\prime}_{0}$.  Here,
$\mu$ is a strictly positive scalar function, assumed continuous,
which has the interpretation of a time-dependent global rate change
for $Q^{}_{0}$, and $C_{\bs{q}(t)}$ is a time-dependent modification
in the form of a traceless equal-input matrix, also assuming that
$\bs{q}$ is continuous.\footnote{The continuity assumptions can later
  be generalised to local integrability of $\mu$ and $\bs{q}$, if one
  replaces the ODE for the flow with the corresponding Volterra
  integral equation. Also, the strict positivity of $\mu$ can be
  slightly relaxed, but we suppress further details on this aspect.}
Here, in line with Fact~\ref{fact:ideal}, we have
\begin{equation}\label{eq:easy}
   Q^{}_{0} C_{\bs{q} (t)} \, = \, \nix \, , \quad
   C_{\bs{q} (t)} C_{\bs{q}(t')} \,  = \, q(t) C_{\bs{q} (t')}
   \, = \, \nix \quad \text{and} \quad
   C_{\bs{q} (t)} \ts Q^{}_{0} \in \cE^{\prime}_{0}
\end{equation}
for all $t,t'\geqslant 0$ under our assumptions. This implies that we
can compute the summands $I_{n} (t)$ of the PBS for the solution of
the Cauchy problem $\dot{M} = M \ts Q$ with $M(0)=\one$ more
explicitly.

Let us begin with the following observation.

\begin{fact}\label{fact:integral}
  Let\/ $\mu$ be a locally integrable function and define\/
  $u(t) = \int_{0}^{t} \mu(\tau) \dd \tau$. Then,
\[
    \int_{0}^{t} u(\tau)^{n} \mu(\tau) \dd \tau \, = \,
    \myfrac{u(t)^{n+1}}{n+1}
\]  
  holds for all\/ $n\in\NN$, and one also obtains the iterated integral
\[
    \int_{0}^{t} \int_{0}^{t_1} \cdots \int_{0}^{t_{n-1}}
    \mu(t^{}_{1}) \ts \mu(t^{}_{2}) \cdots \mu(t^{}_{n}) 
    \dd t^{}_{n} \cdots \dd t^{}_{2} \dd t^{}_{1} \, = \, 
    \myfrac{u(t)^{n}}{n \ts !} \ts .
\]  
\end{fact}

\begin{proof}
  For the first claim, since $u(0)=0$, we employ integration by parts
  to obtain
\[
   \int_{0}^{t} u(\tau)^{n} \mu(\tau) \dd \tau \, = \, u(t)^{n+1}
   - \ts n \int_{0}^{t} u(\tau)^n \mu(\tau) \dd \tau
\]
and solve for the integral, where the integrability of $u^n \mu$ is
clear by standard arguments.

The second claim is true for $n=1$, by the very definition of
$u(t)$. Then, for $n\in\NN$, we get
\[
\begin{split}
    \int_{0}^{t} \int_{0}^{t_1} & \cdots \int_{0}^{t_{n}}
    \mu(t^{}_{1}) \ts \mu(t^{}_{2}) \cdots \mu(t^{}_{n+1}) 
    \dd t^{}_{n+1} \cdots \dd t^{}_{2} \dd t^{}_{1}  \\[2mm]
    & = \int_{0}^{t}\mu(t^{}_{1})\ts \myfrac{u(t^{}_{1})^{n}}{n\ts !}
    \dd t^{}_{1} \, = \, \myfrac{u(t)^{n+1}}{(n+1) !}  \ts ,
\end{split}    
\]
which uses the first claim, and settles the second inductively.
\end{proof}

For the PBS, we have
$I^{}_{1} (t) = \int_{0}^{t} Q(\tau) \dd \tau = u(t) Q^{}_{0} +
C_{\bs{q}^{(1)} (t)}$ with
$\bs{q}^{(1)} (t) = \int_{0}^{t} \bs{q} (\tau) \dd \tau$. Then,
defining
$\bs{q}^{(n+1)} (t) = \int_{0}^{t} \mu(\tau) \ts\ts \bs{q}^{(n)}
(\tau) \dd \tau$ for $n\in\NN$, one inductively finds
\begin{equation}\label{eq:PBS-iter}
  I^{}_{n+1} (t) \, = \, \int_{0}^{t} I^{}_{n} (\tau) \ts Q(\tau) \dd \tau
  \, = \, \myfrac{u(t)^{n+1}}{(n+1)!} \ts\ts Q^{n+1}_{0} +
  C_{\bs{q}^{(n+1)} (t)} \, Q^{n}_{0}
\end{equation}
by an application of Fact~\ref{fact:integral} and the observation that
all other terms vanish as a result of Eq.~\eqref{eq:easy}. Note that
the last term in \eqref{eq:PBS-iter}, for any $t\geqslant 0$, is an 
element of $\cE^{\prime}_{0}$ due to Fact~\ref{fact:ideal}. Putting 
all this together, we get the following result.

\begin{prop}\label{prop:Cauchy-pert}
  Consider the Cauchy problem\/ $\dot{M} = M \ts Q$ with\/ $M(0)=\one$
  for the continuous matrix family defined by Eq.~\eqref{eq:pert-gen}.
  Then, the PBS for its solution has the additive form\/
  $M(t) = \one + A(t) = \one + A^{}_{0}(t) + A_{\tri} (t)$, with\/
  $Q^{0}_{0} = \one$ and
\[
  A^{}_{0} (t) \, = \, \ee^{ u(t) \ts Q^{}_{0}} - \one \qquad
  \text{and} \qquad A_{\tri} (t) \, = \sum_{n=1}^{\infty}
  C_{\bs{q}^{(n)} (t)} \, Q^{n-1}_{0} \ts ,
\]   
where\/ $A^{}_{0} (t) \in \cA^{}_{\ts 0}$ and\/
$A_{\tri} (t) \in \cE^{\prime}_{0}$ for all\/ $t\geqslant 0$. The
infinite sum is compactly converging, and the spectral radius of
$A(t)$, for all sufficiently small\/ $t$, satisfies\/
$\varrho^{}_{\nts A(t)} < 1$.
\end{prop}

\begin{proof}
  One has $M(t) = \one + \sum_{n\geqslant 1} \ts I_n (t)$ from the
  PBS, which is compactly converging by
  \cite[Thm.~1]{BSch}. Evaluating the sum with the terms from
  \eqref{eq:PBS-iter} gives the decomposition into contributions from
  $\cA^{}_{\ts 0}$ and $\cE^{\prime}_{0}$ as stated.
  
  Since $A(0)=\nix$, the claim on the spectral radius is clear by
  continuity.
\end{proof}

Since $u(t) = \cO (t)$ as $t \, {\scriptstyle \searrow} \, 0$, the
result on the spectral radius implies that $M(t)$, at least for all
sufficiently small $t$, possesses a real logarithm in the form of the
convergent series
\[
  R(t) \, = \, \log \bigl( M(t) \bigr) \, = \, \log \bigl( \one + A(t)
  \bigr) \, = \sum_{n=1}^{\infty} \myfrac{(-1)^{n-1}}{n} A(t)^{n} \ts ,
\]  
which defines the principal matrix logarithm. Consequently, we know
that $M(t) = \exp ( R(t) )$, at least for small $t$. Since
$A^{}_{0} (t) \in \cA^{}_{\ts 0}$ and
$A_{\tri} (t) \in \cE^{\prime}_{0}$, we obtain
$A^{}_{0} (t) A_{\tri} (t) = \nix$ and
$A_{\tri} (t)^2 = \nix$ from Fact~\ref{fact:ideal}. This implies
\begin{equation}\label{eq:powers}
    A(t)^{n} \, = \, \bigl( A^{}_{0} (t) + A_{\tri} (t) \bigr)^{n} \, = \,
    A^{}_{0} (t)^{n} + A_{\tri} (t) A^{}_{0} (t)^{n-1}
\end{equation}
for all $n\in\NN$ and $t\geqslant 0$. But this gives
$R(t) = R^{}_{\ts 0} (t) + R_{\tri} (t)$ with
\[
\begin{split}
  R_{\tri} (t) \, & = \sum_{n=1}^{\infty} \myfrac{(-1)^{n-1}}{n}
  A_{\tri} (t) A^{}_{0} (t)^{n-1} \quad \text{and} \\[2mm]
  R^{}_{\ts 0}(t) \, & = \sum_{n=1}^{\infty} \myfrac{(-1)^{n-1}}{n}
  A^{}_{0} (t)^{n} \, = \, \log \bigl( \one + A^{}_{0} (t)\bigr) \, =
  \, u(t) \ts Q^{}_{0}
\end{split} 
\]   
for sufficiently small $t$, where one then always has
$R_{\tri} (t) \in \cE^{\prime}_{0}$. Note that the latter property
will be preserved when an extension of $t$ beyond the circle of
convergence is possible, for instance via analytic continuation.

\begin{coro}\label{coro:log}
  For all sufficiently small\/ $t\geqslant 0$, the solution\/ $M(t)$
  from Proposition~$\ref{prop:Cauchy-pert}$ has the form\/
  $M(t) = \ee^{R(t)}$ with\/ $R(t) = u(t) \ts Q^{}_{0} + R_{\tri} (t)$
  and\/ $R_{\tri} (t) \in \cE^{\prime}_{0}$.  \qed
\end{coro}

To determine $R_{\tri} (t)$, we employ tools from the Magnus
expansion; see the Appendix for a short summary and further
references. In view of the switched order of the matrices in our ODEs
due to the row sum convention for Markov matrices and generators, we
use the (twisted) \emph{adjoint} of two matrices,
\[
    \ad^{}_{A} (B) \, \defeq \, [B,A] \, = \, - [A, B] \ts .
\]
Here, we get
\[
\begin{split}
  \ad^{}_{R(t)} \bigl( Q(t)\bigr) \, & = \, \bigl[ \mu(t) \ts Q^{}_{0}
  + C_{\bs{q} (t)}, u(t) \ts Q^{}_{0} + R_{\tri} (t) \bigr] \\[1mm]
  & = \, -\mu(t) R_{\tri} (t) \ts Q^{}_{0} + u(t)
  \ts C_{\bs{q}(t)} Q^{}_{0}  \, = \, \bigl( u(t) \ts Q (t)
  - \mu(t) \ts R(t) \bigr) Q^{}_{0} \ts ,
\end{split}
\]
where all additional terms in the second step vanish due to
Eq.~\eqref{eq:easy}, while the last step uses
$R^{}_{\ts 0}(t) = u(t)\ts Q^{}_{0}$ from above. Next, we need to
calculate powers of the adjoint, where
$\ad^{n+1}_{A} \defeq \ad^{}_{A} \circ \ad^{n}_{A}$ for
$n\geqslant 1$. One can now repeat the above type of calculation
inductively, with the same kind of cancellations, to obtain
\begin{equation}\label{eq:ad-powers}
  \ad^{n}_{R(t)} \bigl( Q(t)\bigr) \, = \,
  \bigl( u(t) \ts Q(t) - \mu(t) R(t) \bigr)
  u(t)^{n-1} Q^{n}_{0}
\end{equation}
for $n\in\NN$, where it is clear from Fact~\ref{fact:ideal} that this
expression is always an element of $\cE^{\prime}_{0}$, as can be verified
by writing the right-hand side in terms of $C_{\bs{q}(t)}$ and $R_{\tri}(t)$.

Now, with \eqref{eq:ad-powers} and Eq.~\eqref{eq:dot-R} from the
Appendix, we get
\[
  \dot{R} (t) \, = \sum_{n=0}^{\infty} \frac{b_n}{n\ts !} \,
  \ad^{n}_{R(t)} \bigl( Q(t) \bigr) \, = \, Q(t) + \bigl( u(t) \ts
  Q(t) - \mu(t) R(t) \bigr) \frac{f( u(t) \ts Q^{}_{0} ) - \one}{u(t)}
\]
with the well-known meromorphic function
\[
     f(x) \, = \, \frac{x}{\ee^x - 1} \, = 
     \sum_{n=1}^{\infty} \frac{b_n}{n\ts !}  \, x^n ,
     \qquad \text{with} \; f(0) = 0 \ts ,
\]
where the $b_n$ denote the Bernoulli numbers, here with
$b^{}_{1} = - \frac{1}{2}$ and $b^{}_{2m+1}=0$ for all $m\in\NN$. The
power series of $f $ has $2\pi$ as its radius of convergence, so
$f(u(t)\ts Q^{}_{0})$ is well defined for small values of $t$, as is
$\bigl( f(u(t) \ts Q^{}_{0}) - \one\bigr)/u(t)$, where we assume
$u(t) =\int_{0}^{t} \mu (\tau) \dd \tau > 0$ for all $t>0$, which really
is a consequence of our assumptions on $\mu$.

Observing that $R(t) = R^{}_{\ts 0} (t) + R_{\tri} (t)$ with
$\dot{R}^{}_{\ts 0} (t) = \mu(t) \ts Q^{}_{0}$ leads to several
cancellations, one arrives at the inhomogeneous linear ODE
\[
    \dot{R}_{\tri} (t) + \myfrac{\mu(t)}{u(t)} \, R_{\tri} (t) \,
    \bigl( f(u(t)\ts Q^{}_{0}) - \one \bigr) \, = \,  C_{\bs{q}(t)}
    f(u(t) \ts Q^{}_{0} ) \ts ,
\]
which can again be solved by the standard methods used earlier. The
result is
\[
\begin{split}
  R_{\tri} (t) \, & = \int_{0}^{t} \nts C_{\bs{q}(\tau)} f(u(\tau)\ts
  Q^{}_{0}) \exp \Bigl( \, \int_{0}^{\tau} \nts
  \myfrac{\mu(\sigma)}{u(\sigma)} \bigl( f( u(\sigma) \ts Q^{}_{0}) -
  \one \bigr) \dd \sigma \Bigr)
  \dd \tau \\[2mm]
  & \qquad\qquad \cdot\exp \Bigl( - \! \int_{0}^{t} \nts
  \myfrac{\mu(\rho )}{u(\rho )} \bigl(
  f(u(\rho)\ts Q^{}_{0}) - \one \bigr) \dd \rho \Bigr) \\[2mm]
  & = \, \int_{0}^{t} C_{\bs{q}(\tau )} f(u(\tau)\ts Q^{}_{0}) \exp\Bigl(
  \, \int_{0}^{u(\tau)} \nts\nts \myfrac{f(\vartheta \ts
    Q^{}_{0})-\one}{\vartheta} \dd \vartheta\Bigr) \dd \tau \cdot
  \exp\Bigl( - \! \int_{0}^{u(t)} \nts\nts \myfrac{f(\eta \ts
    Q^{}_{0}) - \one}{\eta} \dd \eta \Bigr) ,
\end{split}
\]
where the second expression emerges from a standard substitution,
which is justified when $u$ is strictly increasing.

The expression for $R_{\tri} (t)$ can still be simplified further. One
idea how to achieve this comes from the formula
\begin{equation}\label{eq:Stamm}
     \int_{0}^{t} \myfrac{f(\tau) - 1}{\tau} \dd \tau \, = \, 
     -t + \log\myfrac{\ee^t - 1}{t} \, = \, - t - \log (f(t)) \ts ,
\end{equation}
which holds for $t\geqslant 0$. Carefully inspecting the corresponding
identities for matrix-valued integrals, one can come to the following
helpful identity.

\begin{lemma}\label{lem:integral}
  Let\/ $B \in \Mat(d,\RR)$ be fixed. Then, for sufficiently small\/
  $t\geqslant 0$, one has
\[
  \int_{0}^{t} \myfrac{f(\tau B) - \one}{\tau} \dd \tau \, = \,
  \log \bigl( f (t B)^{-1} \ee^{-t B} \bigr) ,
\]    
   where\/ $\log$ refers to the principal matrix logarithm.
\end{lemma}

\begin{proof}
  Clearly, both sides evaluate to $\nix$ for $t=0$, and the spectrum
  of $t B$ lies inside the circle of convergence for the power series
  of $f$, as long as $t$ is small enough. Also, again for sufficiently
  small $t$, the matrix argument of the logarithm cannot have any
  negative eigenvalues, so the principal logarithm is well defined.

  Now, we need to check that both sides have the same derivative. On
  the left, one has
\[
    \myfrac{1}{t} \bigl( f(t B) - \one \bigr) \, = \,
    \myfrac{f(z) - 1}{z} \Big|_{z = t B} \cdot B \, = \,
     \myfrac{1+z-\ee^{z}}{z (\ee^{z}-1)}\Big|_{z=t B} \cdot B \ts ,
\]
with the usual understanding of matrix functions via converging power
series.  In comparison, the right-hand side gives
\[
   \myfrac{\dd}{\dd z} \log \myfrac{\ee^{-z}}{f(z)} \Big|_{z=t B} \cdot B \, = \,
   \myfrac{z \, \ee^{z}}{\ee^{z}-1} \myfrac{\dd}{\dd z} \myfrac{1-\ee^{-z}}{z}
   \Big|_{z=t B} \cdot B \, = \, 
   \myfrac{1+z-\ee^{z}}{z (\ee^{z}-1)}\Big|_{z=t B} \cdot B \ts ,
\]
which agrees with the previous expression and completes the argument.
\end{proof}

Using $f(t)\ee^t = f(-t)$ and inserting the formula from
Lemma~\ref{lem:integral} into our solution for $R_{\tri} (t)$ then
gives the significantly simpler expression
\begin{equation}\label{eq:much-simpler}
    R_{\tri}(t) \, = \int_{0}^{t} \! C_{\bs{q}(\tau)} \ee^{-u(\tau) \ts Q^{}_{0}}
    \dd \tau \cdot  f \bigl( - u(t) \ts Q^{}_{0}\bigr) .
\end{equation}
Let us sum up as follows.

\begin{theorem}\label{thm:general}
  Consider the Cauchy problem\/ $\dot{M} = M \ts Q$ with\/ $M(0)=\one$
  for the matrix family\/ $Q(t) = \mu(t) \ts Q^{}_{0} + C_{\bs{q}(t)}$
  from \eqref{eq:pert-gen}, where\/ $\mu$ is a positive scalar
  function and\/ $\bs{q}$ an\/ $\RR^d$-valued function with\/ $q(t)=0$
  for all\/ $t\geqslant 0$, both assumed continuous.  This problem has
  a unique solution in the form of the PBS.  For small enough\/ $t$,
  it also satisfies\/ $M(t) = \ee^{R(t)}$ with\/
  $R(t) = u(t) \ts Q^{}_{0} + R_{\tri} (t)$ and the\/
  $R_{\tri} (t) \in \cE^{\prime}_{0}$ from \eqref{eq:much-simpler},
  where\/ $u(t) = \int_{0}^{t} \mu (\tau) \dd \tau$.
    
  Further, when all\/ $Q(t)$ are Markov generators, the set\/
  $\{ M(t) : t\geqslant 0 \}$ defines a flow of Markov matrices. For
  all sufficiently small\/ $t$, the real logarithm\/ $R(t)$ is then a
  Markov generator, which means that\/ $M(t)$, for any fixed\/ $t$, is
  a Markov matrix that is also embeddable into a time-homogeneous
  Markov semigroup.  \qed
\end{theorem}  

Let us comment on the condition with small $t$. In our derivation, we
have used the approach to matrix functions via convergent Taylor
series as in \cite[Thm.~4.7]{Higham}, which limits the eigenvalues of
the matrix argument in $f(\pm \ts u(t) \ts Q^{}_{0})$ to be less than
$2 \pi$ in modulus. However, $f$ is a meromorphic function on $\CC$,
with poles of first order on the imaginary axis, namely at
$2\ts n\pi\ii$ for all $0\ne n\in\ZZ$. As long as the spectrum of
$u(t) \ts Q^{}_{0}$ avoids these poles (which is the generic case),
the solution formula \eqref{eq:much-simpler} remains valid, as can be
seen via analytic continuation.

Let us make two consistency calculations. First, with
$Q^{}_{0} = \nix$ in \eqref{eq:pert-gen}, we get
$Q(t) \in \cE^{\prime}_{0}$ for all $t\geqslant 0$, and
\eqref{eq:much-simpler} simplifies to
\[
  R_{\tri} (t) \, = \int_{0}^{t} C_{\bs{q} (\tau)} \dd \tau \, = \,
  C_{\bs{x} (t)}
\]
with $\bs{x} (t) = \int_{0}^{t} \bs{q} (\tau) \dd \tau$. Since
$R(t) = R_{\tri} (t)$ in this case, we are back to
Eq.~\eqref{eq:trivial-case}.

A little less obvious is the case where $Q^{}_{0}$ is a constant-input
matrix.  Let us thus consider
$Q(t) = \mu(t)\ts Q^{}_{0} + C_{\bs{r} (t)}$ with
$Q^{}_{0} = J^{}_{d}$ and $r(t) \equiv 0$. Then,
$Q^{2}_{0} = - Q^{}_{0}$ and
$C_{\bs{r} (t)} Q^{}_{0} = - C_{\bs{r} (t)}$, both as a result of
Eq.~\eqref{eq:Q-product}. Consequently, for any function $\phi$ that
is analytic around $0$ and all $t,s \geqslant 0$, we get
\[
    C_{\bs{r} (t)}\ts  \phi \bigl(u(s)\ts Q^{}_{0} \bigr) \, = \, 
    \phi \bigl(- u(s) \bigr) \ts C_{\bs{r} (t)} \ts ,
\]
with the standard restriction on the arguments of $\phi$ in relation
to its radius of convergence, which can later be lifted by analytic
continuation.  With this, the formula for $R_{\tri} (t)$ gives
\[
\begin{split}
  R_{\tri} (t) \, & = \int_{0}^{t} C_{\bs{r}(\tau)} \exp \bigl(-
  u(\tau) \ts Q^{}_{0} \bigr)
  \dd \tau \cdot f \bigl( - u(t) \ts Q^{}_{0} \bigr) \\[2mm]
  & = \int_{0}^{t} \ee^{u(\tau)} C_{\bs{r}(\tau)} f \bigl(- u(t) \ts
  Q^{}_{0} \bigr) \dd \tau \, = \, f(u(t)) \int_{0}^{t} \ee^{u(\tau)}
  C_{\bs{r} (\tau)} \dd \tau \ts .
\end{split}    
\]
Since $R^{}_{0} (t) = u(t)\ts Q^{}_{0}$ and 
\[
     u(t) \, = \, f(u(t)) \bigl( \ee^{u(t)} - 1 \bigr) \, = 
    f(u(t)) \int_{0}^{t} \ee^{u(\tau)} \mu(\tau) \dd \tau \ts ,
\]
we thus get the simplified formula
\[
    R(t) \, = \, f(u(t)) \int_{0}^{t} \ee^{u(\tau)}
    \bigl( \mu(\tau) \ts Q^{}_{0} + C_{\bs{r} (\tau)} \bigr) \dd \tau
    \, = \, f(u(t)) \int_{0}^{t} \ee^{u(\tau)} Q(\tau) \dd \tau \ts ,
\]
which agrees with Eq.~\eqref{eq:weighted-integral}, as it must.
\smallskip

It is possible to generalise our result from Theorem~\ref{thm:general}
even a little further. To do so, we consider the (assumed continuous)
matrix family
\begin{equation}\label{eq:further}
    Q(t) \, = \, Q^{}_{0} (t) + C_{\bs{q}(t)}
\end{equation}
with $q(t) = \tr (C_{\bs{q}(t)}) \equiv 0$, where the
$Q^{}_{0}(t)\in\cA^{}_{\ts 0}$ define a commuting matrix family, that
is, they satisfy $[Q^{}_{0} (t), Q^{}_{0} (s)]=\nix$ for all
$t,s\geqslant 0$. We now define the integral
$R^{}_{0} (t) = \int_{0}^{t} Q^{}_{0} (\tau) \dd \tau$, which
satisfies $R^{}_{0} (0)=\nix$ and $[R^{}_{0} (t), Q^{}_{0} (s)]=\nix$
for all $t,s \geqslant 0$. The PBS for the Cauchy problem with the
matrix family from \eqref{eq:further} leads to
$I^{}_{1} (t) = R^{}_{0} (t) + Q^{(1)}_{\tri} (t)$, where one has
$Q^{(1)}_{\tri} (t) = \int_{0}^{t} C_{\bs{q}(\tau)} \dd \tau \in
\cE^{\prime}_{0}$, and then inductively to
\[
    I^{}_{n+1} (t) \, = \, \myfrac{R^{n+1}_{0} (t)}{(n+1)!}
    + Q^{(n+1)}_{\tri} (t)  \quad \text{with} \quad Q^{(n+1)}_{\tri} (t)
    \, = \int_{0}^{t} Q^{(n)}_{\tri} (\tau) \ts Q^{}_{0} (\tau) \dd \tau
    \in \cE^{\prime}_{0}
\]
for $n\in\NN$. The derivation of the leading term uses the integration
identity
\[
     \int_{0}^{t} R^{n}_{0} (\tau) \ts Q^{}_{0} (\tau) \dd \tau \, = \,
     \myfrac{R^{n+1}_{0} (t)}{n+1} \ts ,
\]
which holds for $n\in\NN_{0}$ and is a matrix-valued analogue of
Fact~\ref{fact:integral}. As $[R^{}_{0} (t), Q^{}_{0} (t)]=\nix$, the
claimed formula can be verified by differentiation, with
$\dot{R}^{}_{0} = Q^{}_{0}$.  As before, this leads to the solution of
the Cauchy problem in the form
$M(t) = \one + A^{}_{0} (t) + A_{\tri} (t)$ with
$A^{}_{0} (t) = \ee^{R^{}_{0} (t)} - \one$ and
$A_{\tri} (t) = \sum_{n\geqslant 1} Q^{(n)}_{\tri} (t)$, which is
compactly converging. Here, one has $A^{}_{0} (t) \in \cA^{}_{\ts 0}$
and $A_{\tri} (t) \in \cE^{\prime}_{0}$ for all $t\geqslant 0$, and we
get the same formula for the powers as in Eq.~\eqref{eq:powers}.

Setting $M(t) = \exp\bigl( R(t)\bigr)$, which is at least possible for
small $t$, one gets the decomposition
$R(t) = R^{}_{0} (t) + R_{\tri}(t)$, compare Corollary~\ref{coro:log},
again with
$R_{\tri} (t) = \sum_{n\in\NN} \frac{(-1)^{n-1}}{n} A_{\tri}(t)
A^{n-1}_{0} (t)$. Dropping the notation indicating explicit time
dependence for a moment, one can now calculate the iterated adjoint
using exactly the same algebraic steps as around
\eqref{eq:much-simpler} together with
$[R^{}_{0} (t), Q^{}_{0} (t)]=\nix$. This leads to
\[
   \ad^{\ts n}_{R} (Q) \, = \, (Q_{\bs{q}} R^{}_{0} - R_{\tri} Q^{}_{0}) 
   R^{n-1}_{0} \, = \, (Q R^{}_{0} - R \ts\ts Q^{}_{0}) R^{n-1}_{0}
\]
for $n\in\NN$, together with $\ad^{\ts 0}_{R} (Q) = Q$. This gives the
ODE for $R$ via Eq.~\eqref{eq:dot-R} as
\[
    \dot{R} \, = \sum_{n=1}^{\infty} \myfrac{b_n}{n\ts !} \, \ad^{n} _{R} (Q)
    \, = \, Q^{}_{0} + C_{\bs{q}} \ts f( R^{}_{0})
    - R_{\tri} Q^{}_{0} \ts\ts g(R^{}_{0} ) \ts ,
\]
with $g(z) = \frac{f(z)-1}{z}$. This $g$ is a meromorphic function on
$\CC$ with simple poles at the same places as $f$, and with
$g(0)=-\frac{1}{2}$.  Since $\dot{R}^{}_{0} = Q^{}_{0}$, we can split
off the part for $R_{\tri}$, which gives the inhomogeneous linear ODE
\[
    \dot{R}_{\tri} + R_{\tri} \ts Q^{}_{0} \ts\ts g(R^{}_{0})
    \, = \, C_{\bs{q}} \ts f(R^{}_{0}) \ts ,
\]
with the solution
\[
   R_{\tri} (t) \, = \int_{0}^{t} \! C_{\bs{q}(\tau)} \ts
   f\bigl( R^{}_{0} (\tau ) \bigr) \exp \Bigl(\, \int_{0}^{\tau}\!\!
   Q^{}_{0} (\sigma) \ts\ts g \bigl( R^{}_{0} (\sigma) \bigr) \dd
   \sigma \Bigr) \dd \tau \cdot \exp \Bigl(- \! \int_{0}^{t}\!
   Q^{}_{0} (\rho) \ts\ts g \bigl( R^{}_{0} (\rho) \bigr) \dd
   \rho \Bigr) .
\]
With $h(z) = \int_{0}^{z} g(x) \dd x = - z - \log ( f(z))$, which
satisfies $h(0)=0$, and Eq.~\eqref{eq:Stamm}, we then get
\[
    \int_{0}^{t} Q^{}_{0} (\tau) \ts\ts g\bigl( R^{}_{0} (\tau) \bigr)  
    \dd \tau \, = \, h \bigl( R^{}_{0} (t)\bigr) ,
\]
as can easily be verified by differentiation. This is a variant of
Lemma~\ref{lem:integral}, which now gives
\[
    \exp \bigl( {\pm \ts h (R^{}_{0} (t))} \bigr) \, = \, 
    \exp \bigl( {\mp \ts R^{}_{0} (t)} \bigr) f ( R^{}_{0} 
    (t) )^{\mp\ts 1} .
\]
Inserting this into the solution formula for $R_{\tri} (t)$,
and using $f(x) \ee^x = f(-x)$ again, we obtain
\begin{equation}\label{eq:final-R}
   R_{\tri} (t) \, = \int_{0}^{t} C_{\bs{q} (\tau)} \ee^{- R^{}_{0} (\tau)} 
   \dd \tau \cdot f\bigl( - R^{}_{0} (t)\bigr) .
\end{equation}
In principle, $R_{\tri} (t)$ is of the form $C_{\bs{x} (t)}$ for some
suitable function $\bs{x}$. However, no simple general formula seems
possible, unless one makes further assumptions on $Q^{}_{0}$.  We have
thus derived the following extension of Theorem~\ref{thm:general}.

\begin{coro}
  Consider the Cauchy problem\/ $\dot{M} = M \ts Q$ with\/ $M(0)=\one$
  for the matrix family\/ $\{ Q(t): t \geqslant 0 \}$ of
  Eq.~\eqref{eq:further}, with\/ $Q^{}_{0} (t)$ being commuting
  matrices from\/ $\cA^{}_{\ts 0}$ and\/ $C_{\bs{q} (t)}$ having zero
  trace. Then, at least for small\/ $t$, the solution is of the form\/
  $M(t) = \exp ( R(t) )$ with\/
  $R(t) = R^{}_{0} (t) + R_{\tri} (t)$, where\/
  $R^{}_{0} (t) = \int_{0}^{t} Q^{}_{0} (\tau) \dd \tau$ and\/
  $R_{\tri} (t)$ is the matrix from \eqref{eq:final-R}.  \qed
\end{coro}

One obvious special case emerges via $\bs{q}(t)\equiv \bs{0}$. Then,
we have $R_{\tri} \equiv \nix$, and the solution boils down to
\[
    M(t) \, = \exp \bigl( R^{}_{0} (t) \bigr) \, = \,
    \exp \Bigl( \, \int_{0}^{t} Q^{}_{0} (\tau) \dd \tau \Bigr)
\]
as it must, because this is the easy case of commuting matrices.

%%%%%%%%%%%%%%%%%%%%%%%%%%%%%%%%%
%%%%%%%%%%%%%%%%%%%%%%%%%%%%%%%%%

\appendix
\section{Peano--Baker series and Magnus expansion}

Here, we give a quick summary of two helpful tools for the solution of
non-autonomous matrix-valued Cauchy problems of the form
\begin{equation}\label{eq:A1}
     \dot{X} (t) \, = \, X(t) \ts A(t) \quad \text{with} \quad 
     X(t^{}_{0}) \, = \, X^{}_{0} \ts ,
\end{equation}
where we consider the forward flow for $t\geqslant t^{}_{0}$. We use
this version (with $A(t)$ on the right) in view of the applications to
Markov flows.  In the simple case where $[ A(t),A(s)]=\nix$ for all
$t,s\geqslant t^{}_{0}$, which includes the case that $A(t)$ is a
constant matrix, one gets the solution as
\begin{equation}\label{eq:A2}
  X(t) \, = \, X^{}_{0} \exp \Bigl( \, \int_{t^{}_{0}}^{t}
  A(\tau) \dd \tau \Bigr),
\end{equation}
with uniqueness under the usual constraints on $A(t)$. This solution
is natural from the \emph{Volterra integral equation} point of view,
where \eqref{eq:A1} is replaced by the integral version
\begin{equation}\label{eq:A3}
  X(t) \, = \, X^{}_{0} + \int_{t^{}_{0}}^{t} X(\tau)
  \ts A(\tau) \dd \tau \ts .
\end{equation}
When $A(t) = A$ for all $t$, the standard \emph{Picard
  iteration} leads to the well-known formula
\[
  X(t) \, = \, X^{}_{0} \Bigl( \one + (t-t^{}_{0}) A +
  \myfrac{(t-t^{}_{0})^2}{2} A^2 + \myfrac{(t-t^{}_{0})^3}{6} A^3 +
  \ldots \Bigr) \, = \, X^{}_{0} \ee^{(t-t^{}_{0}) A},
\]
which is a simple special case of \eqref{eq:A2}.

The situation becomes more complicated when the matrices $A(t)$ no
longer commute. There are still two helpful approaches, namely the
Peano--Baker series (PBS, which emerges from a careful application of
the Picard iteration) and the Magnus expansion (ME, which is related
to the Baker--Campbell--Hausdorff formula and employs some techniques
from Lie theory).

Let us begin with the PBS. Here, one finds a compactly converging
series representation of the solution in the form
$X(t) = X^{}_{0} \cdot \sum_{n=0}^{\infty} \, I_{n} (t) $ with
\[
   I^{}_{0} (t) \, = \, \one
   \quad \text{and} \quad I_{n+1} (t) \, = \int_{t^{}_{0}}^{t}
   I_{n} (\tau) A(\tau) \dd \tau \quad \text{for } n\in\NN_{0} \ts ;
\] 
see \cite{BSch} and references therein for background and proofs, and
\cite{BS3} for a formulation in our present context and a comparison
with the time-ordered exponential used in physics, where it is
sometimes called the \emph{Dyson series}.  While it is rare that one
can calculate the $I_n$ explicitly (unless the $A(t)$ commute), the
PBS is still useful for structural insight, in particular if the
matrices $A(t)$ come from an algebra.
\smallskip

Another tool is the ME, which approaches linear ODEs via an
exponential solution; see \cite{ME} and references therein for an
extensive exposition. Here, we again look at the matrix-valued 
Cauchy problem from \eqref{eq:A1}
where the $A(t)$ constitute some (sufficiently nice) matrix family,
but need \emph{not} commute. The flow, at least for small times, is of
the form $X(t) = X^{}_{0} \exp\bigl( R (t^{}_{0} , t)\bigr)$, where we
now simply write $R(t)$ and take $t^{}_{0} = 0$ for convenience.

The Poincar\'{e}--Hausdorff identity from matrix groups now states
that 
\[
   \ee^{-R (t)} \myfrac{\dd}{\dd t} \ee^{R(t)} \, = \, A (t) \ts ,
\]
from the left-hand side of which one can derive an expression for
$\dot{R} (t)$ in the form
\begin{equation}\label{eq:dot-R}
  \dot{R}(t) \, = \, f \bigl(\ad^{}_{R(t)} \bigr) \bigl( A (t)\bigr) \ts ,
\end{equation}
where $f$ is the meromorphic function defined by
$f(x) = \frac{x}{\ee^{x}-1}$ and $\ad^{}_{C} (B) \defeq [B,C]$. This
unusual (twisted) version of the standard adjoint,
$\mathrm{ad}^{}_{C} (B) = [C,B]$, is taken to match the order of
matrix multiplication we have used in the ODE \eqref{eq:A1}. Then,
powers of $\ad$ are defined recursively, so $\ad^{\ts 0}_{C} (B) = B$
and $\ad^{n+1}_{C} (B) = \bigl[ \ad^{n}_{C} (B), C \bigr]$ for
$n\geqslant 0$. The operator in \eqref{eq:dot-R} is then defined via
the power series of $f$ around $0$, which is
\[
     f(x) \, = \sum_{n=0}^{\infty} \myfrac{b_n}{n\ts !} \, x^n \, = \, 
     1 - \myfrac{x}{2}  + \myfrac{x^2}{12} + \cO (x^4) \ts ,
\]
where the $b_n$ are the Bernoulli numbers.

While the right-hand side of \eqref{eq:dot-R} usually cannot be
calculated in closed terms, and is then employed approximately via
suitable truncations, this paper has examined several special
situations where the series can be worked out, and then admits the
exact computation of $R(t)$ via an explicit integration step.

\section*{Acknowledgements}

The authors thank an anonymous referee for a number of useful
suggestions, which helped us to improve our presentation. This work
was supported by the German Research Foundation (DFG, Deutsche
Forschungsgemeinschaft) under \mbox{SFB 1283/2 2021 -- 317210226}.

\end{document}